  \documentclass[12pt,letterpaper,oneside]{amsart}

  \usepackage{amsmath,amssymb,amsthm}
  \usepackage[margin=1.1in]{geometry}
  \usepackage{xspace,xcolor}
  \usepackage{graphicx}

  \usepackage[
     breaklinks,colorlinks,
     citecolor=blue,linkcolor=red,urlcolor=teal
  ]{hyperref}

  \newcommand{\calA}{\mathcal{A}}
  
  \newcommand{\calC}{\mathcal{C}}

  \newcommand{\calG}{\mathcal{G}}

  \newcommand{\calK}{\mathcal{K}}

  \newcommand{\calN}{\mathcal{N}}
  
  \newcommand{\calP}{\mathcal{P}}
  \newcommand{\calQ}{\mathcal{Q}}

  \newcommand{\calT}{\mathcal{T}}

  \newcommand{\calX}{\mathcal{X}}
  \newcommand{\calY}{\mathcal{Y}}

  \newcommand{\FF}{\mathbb{F}}

  \newcommand{\RR}{\mathbb{R}}

  \newcommand{\gothic}{\mathfrak}
  \newcommand{\ga}{{\gothic a}}
  \newcommand{\gb}{{\gothic b}}

  \newtheorem{theorem}{Theorem}[section]
  \newtheorem{proposition}[theorem]{Proposition}
  \newtheorem{corollary}[theorem]{Corollary}
  \newtheorem{lemma}[theorem]{Lemma}

  \newtheorem{introthm}{Theorem}
  \newtheorem{introcor}[introthm]{Corollary}

  \theoremstyle{definition}
  \newtheorem{definition}[theorem]{Definition}
  
  \newtheorem*{claim*}{Claim}

  \newtheorem*{question*}{Question}
  \newtheorem*{answer*}{Answer}
  \newtheorem*{application*}{Application}

  \theoremstyle{remark}
  \newtheorem{remark}[theorem]{Remark}
  \newtheorem*{remark*}{Remark}

  \newcommand{\secref}[1]{Section~\ref{Sec:#1}}
  \newcommand{\thmref}[1]{Theorem~\ref{Thm:#1}}
  
  \newcommand{\lemref}[1]{Lemma~\ref{Lem:#1}}
  \newcommand{\propref}[1]{Proposition~\ref{Prop:#1}}
  
  \newcommand{\remref}[1]{Remark~\ref{Rem:#1}}

  \newcommand{\figref}[1]{Figure~\ref{Fig:#1}}
  \newcommand{\defref}[1]{Definition~\ref{Def:#1}}
  
  \newcommand{\eqnref}[1]{Equation~\ref{Eq:#1}}

  \DeclareMathOperator{\twist}{twist}

  \DeclareMathOperator{\Ext}{Ext}
  
  \DeclareMathOperator{\diam}{diam}

  \DeclareMathOperator{\I}{i}
  
  \DeclareMathOperator{\Area}{Area}

  \newcommand{\emul}{\stackrel{{}_\ast}{\asymp}}
  \newcommand{\gmul}{\stackrel{{}_\ast}{\succ}}
  \newcommand{\lmul}{\stackrel{{}_\ast}{\prec}}
  \newcommand{\eadd}{\stackrel{{}_+}{\asymp}}
  \newcommand{\gadd}{\stackrel{{}_+}{\succ}}
  \newcommand{\ladd}{\stackrel{{}_+}{\prec}}

  \newcommand{\Aut}{\ensuremath{\operatorname{Aut}}\xspace} 
  \newcommand{\Out}{\ensuremath{\operatorname{Out}}\xspace} 
  \newcommand{\Inn}{\ensuremath{\operatorname{Inn}}\xspace} 

  \newcommand{\s}{\ensuremath{S}\xspace} 
  \newcommand{\T}{\ensuremath{\calT(S)}\xspace} 
    
  \newcommand{\cc}{\ensuremath{\calC(S)}\xspace} 

  \newcommand{\Teich}{{Teichm\"uller }} 

  \newcommand{\param}{{\mathchoice{\mkern1mu\mbox{\raise2.2pt\hbox{$
  \centerdot$}}
  \mkern1mu}{\mkern1mu\mbox{\raise2.2pt\hbox{$\centerdot$}}\mkern1mu}{
  \mkern1.5mu\centerdot\mkern1.5mu}{\mkern1.5mu\centerdot\mkern1.5mu}}}

  \renewcommand{\setminus}{{\smallsetminus}}
  
  \newcommand{\st}{\mathbin{\mid}} 
   
  \newcommand{\from}{\colon\thinspace} 
  \newcommand{\ep}{\epsilon}

  \newcommand{\GT}{{\calG_T}} 
  \newcommand{\GL}{{\calG_L}} 
  \newcommand{\dL}{{d_L}} 
  \newcommand{\dT}{{d_T}}

  \newcommand{\balpha}{{\overline \alpha}}
  \newcommand{\bbeta}{{\overline \beta}}

\begin{document}


  \title    {Bounded combinatorics and the Lipschitz metric on
             Teichm\"uller space}
  \author   {Anna Lenzhen}
  \author   {Kasra Rafi}
  \author   {Jing Tao}
  

  \begin{abstract} 
    
    Considering the \Teich space of a surface equipped with Thurs\-ton's
    Lipschitz metric, we study geodesic segments whose endpoints have
    bounded combinatorics. We show that these geodesics are cobounded, and
    that the closest-point projection to these geodesics is strongly
    contracting. Consequently, these geodesics are stable. Our main tool is
    to show that one can get a good estimate for the Lipschitz distance by
    considering the length ratio of finitely many curves.

  \end{abstract}
  
  \maketitle
  

\section{Introduction}
 
  \label{Sec:Intro}

  Let \T be the \Teich space of a surface \s of finite type, that is, the
  space of marked hyperbolic (or conformal) structures on \s. In
  \cite{thurston:MSM}, Thurston introduced an asymmetric metric $\dL$ for
  \T which we refer to as the Lipschitz metric. For two marked hyperbolic
  structures $x$ and $y$,  $\dL(x, y)$ is defined to be the logarithm of
  the infimum of Lipschitz constants of any homeomorphism from $x$ to $y$
  that is homotopic to the identity. The geometry of the Lipschitz metric
  is very rich, as Thurston shows in his paper. However, many aspects of it
  remain unexamined. 
  
  It is known that \Teich space equipped with the \Teich metric or the
  Lipschitz metric is not Gromov hyperbolic because the thin parts of \T
  have a product like structure (See \cite{minsky:PR} and \cite{rafi:LT}).
  However, certain \Teich geodesics exhibit behaviors that resemble that of
  geodesics in a hyperbolic space. Namely, the closest-point projection to
  these geodesics is strongly contracting (\cite{minsky:QP}). In this
  paper, we investigate whether a similar phenomenon is also present in the
  Lipschitz metric. 
   
  We use tools that have been developed and successfully applied in the
  study of \Teich geodesics, namely the curve graphs of different
  subsurfaces of \s. When $x$ is in the thick part, the geometry of $x$ can
  be coarsely encoded by its associated short marking $\mu_x$, which is a
  finite collection of simple closed curves. Given $x, y \in \T$, there are
  many results relating the combinatorics of markings $\mu_x$ and $\mu_y$
  to the behavior of the \Teich geodesic connecting $x$ and $y$. (See
  \cite{rafi:SC, rafi:CM, rafi:LT}, or \cite{rafi:HT} for a review of some
  of these results in one paper.) 
   
  Contrasting with the \Teich metric, there is no unique geodesic in the
  Lipschitz metric from $x$ to $y$. But one hopes that qualitative
  information about a Lipschitz geodesic can still be extracted from the
  end markings $\mu_x$ and $\mu_y$. The first natural situation to consider
  is when $\mu_x$ and $\mu_y$ have bounded combinatorics. That is when, for
  every proper subsurface $Y$ of \s, the distance $d_Y(\mu_x,\mu_y)$ in the
  curve graph of $Y$ between the projections of $\mu_x$ and $\mu_y$ to $Y$
  is uniformly bounded (see \defref{BC}). For the \Teich metric, this is in
  fact equivalent to the \Teich geodesic between $x$ and $y$ being
  cobounded (see \cite{rafi:SC} and \cite{rafi:HT}).   

  Our first result is that if $\mu_x$ and $\mu_y$ have bounded
  combinatorics, then every Lipschitz geodesic from $x$ to $y$ is
  cobounded. In fact, they are all well approximated by the unique \Teich
  geodesic connecting $x$ and $y$.
  
  \begin{introthm}[Bounded combinatorics implies cobounded]
  \label{Thm:IntroFellow}
  
     Assume, for $x,y \in \T$ in the thick part of \Teich space, that
     $d_Y(\mu_x, \mu_y)$ is uniformly bounded for every proper subsurface
     $Y \subset S$. Then any geodesic $\GL$ in the Lipschitz metric
     connecting $x$ to $y$ fellow travels the \Teich geodesic $\GT$ with
     endpoints $x$ and $y$. Consequently, $\GL$ is cobounded. 
  
  \end{introthm}
  
  To restate \thmref{IntroFellow} more succinctly is to say that $\GT$,
  viewed as a set in the Lipschitz metric, is quasi-convex. A standard
  argument for showing a set is quasi-convex is to show that the
  closest-point projection to the set is strongly contracting. Indeed, this
  is how we prove \thmref{IntroFellow}. 
  
  \begin{introthm}[Lipschitz projection to \Teich geodesics]
  \label{Thm:IntroProjection} 
  
     Let $\GT$ be a cobounded Teich\-m\"uller geodesic. Then the image of a
     Lipschitz ball disjoint from $\GT$ under the closest-point projection
     to $\GT$ (with respect to the Lipschitz metric) has uniformly bounded
     diameter. That is, the closest-point projection to $\GT$ is strongly
     contracting.
   
  \end{introthm}
  
  This is analogous to Minsky's theorem (\cite{minsky:QP}) that the
  closest-point projection in the \Teich metric to a cobounded \Teich
  geodesic is strongly contracting. Combining \thmref{IntroFellow} and
  \thmref{IntroProjection}, we obtain:

  \begin{introthm}[Strongly contracting for projections to Lipschitz
  geodesics]\label{Thm:IntroContracting}

     Suppose $\GL$ is a Lipschitz geodesic whose endpoints have bounded
     combinatorics. Then the closest-point projection to $\GL$ is strongly
     contracting.  
  
  \end{introthm}
  
  \thmref{IntroContracting} is a negative-curvature phenomenon. A natural
  consequence is stability of $\GL$. In other words,

  \begin{introcor}[Stability of Lipschitz geodesics]
  \label{Cor:IntroStable} 
       
     If $\GL$ is a Lipschitz geodesic whose endpoints have bounded
     combinatorics, then any quasi-geodesic with the same endpoints as
     $\GL$ fellow travels $\GL$.
  
  \end{introcor}

  It would be interesting to know whether the converse of
  \thmref{IntroFellow} holds. In the \Teich metric, a geodesic stays in the
  thick part if and only if the endpoints have bounded combinatorics.
  However, this seems not to be the case for the Lipschitz metric. We
  investigate the behavior of a Lipschitz geodesic where the endpoints do
  not necessarily have bounded combinatorics in a subsequent paper. 

  \subsection*{Summary of the proofs} 

  We use the detour through a \Teich geodesic for two reasons. First,
  because it is already established that $\GT$ is cobounded if and only if
  the endpoints have bounded combinatorics. But also because the lengths of
  curves (both hyperbolic length and extremal length) along a cobounded
  \Teich geodesic are known to behave like a $\cosh$ function; the length
  of a curve $\alpha$ is minimal at the balanced point $\GT(t_\alpha)$ and
  grows exponentially fast in both directions. 
  
  Our proof of \thmref{IntroProjection} is to a large extent inspired by
  Minsky's proof in the \Teich setting. However, the following crucial
  length estimate used by Minsky has no analogue in our setting. Given a
  curve $\alpha$ and $x \in \T$, let $\Ext_x(\alpha)$ and $\ell_x(\alpha)$
  denote respectively the extremal length and the hyperbolic length of
  $\alpha$ in $x$. For every two curves $\alpha$ and $\beta$, Minsky showed
  that
  \begin{equation} \label{Eq:Minsky} 
     \Ext_x(\alpha) \Ext_x(\beta) \geq \I(\alpha,\beta)^2,
  \end{equation}
  where $\I(\alpha,\beta)$ is the geometric intersection number between
  $\alpha$ and $\beta$. While the \Teich distance is computed using
  extremal length ratios of curves (\eqnref{Kerck}), the Lipschitz distance
  is computed using hyperbolic length ratios (\eqnref{Thurston}). However,
  there is no analogue of \eqnref{Minsky} for hyperbolic length. For $x$ in
  the thin part of \Teich space, the product $\ell_x(\alpha) \ell_x(\beta)$
  can be arbitrarily close to zero, while $\I(\alpha,\beta)$ can be
  arbitrarily large.
   
  Our approach to the proof of \thmref{IntroProjection} is to give an
  effective description of the closest-point projection $\pi_\GT(x)$ of a
  point $x\in \T$ to a \Teich geodesic $\GT$  (the closest-point
  projection is with respect to the Lipschitz metric). Let $\mu_x$ be a
  short marking on $x$. Then $\pi_\GT(x)$ is near $\GT(t_\alpha)$, where
  $t_\alpha$ is the balanced time of a curve $\alpha \in \mu_x$ (see
  \lemref{Balanced}). This follows from the $\cosh$--like behavior of
  lengths along a \Teich geodesic and the following:
  
  \begin{introthm}[Candidate curves]\label{Thm:IntroMarking}
     
     For $x, y \in \T$, we have
     \[ 
       \dL(x,y) \eadd \max_{\alpha \in \mu_x} \log
       \frac{\ell_y(\alpha)}{\ell_x(\alpha)}, 
     \] 
     where $\dL(x,y)$ is the Lipschitz distance from $x$ to $y$ and $\eadd$
     means equal up to an additive error depending only on the topology of
     $S$.

  \end{introthm}
  
  A special case of \thmref{IntroMarking} where $x$ and $y$ are assumed to
  be in the thick part of \T was done in \cite{rafi:TL}. Thurston's
  formula (\eqnref{Thurston}) for the Lipschitz distance implies that there
  is some curve $\alpha$ such that $\log \frac{\ell_y (\alpha)}{\ell_x
  (\alpha)}$ is a good estimate for $\dL(x,y)$. \thmref{IntroMarking}
  implies that, to find such an $\alpha$, one only needs to examine the
  finitely many curves that appear in $\mu_x$. We will call such a curve
  $\alpha$ in $\mu_x$ a \emph{candidate curve} from $x$ to $y$.
  
  The proof of \thmref{IntroMarking} requires some way of estimating the
  hyperbolic length of a curve in terms of a marking on $S$. We derive two
  formulas for this purpose and their proofs take up a large part of the
  paper. The first formula allows us to estimate, up to a multiplicative
  error, the length of any curve $\gamma$ by a linear sum of the lengths of
  the curves in a short marking, with coefficients coming from the
  intersection of $\gamma$ with the curves in the marking
  (\propref{ShortEquality}). The proof relies on the geometry of the
  thick-thin decomposition of a hyperbolic surface. The second formula uses
  a topological argument to show that, if the short marking is replaced by
  an arbitrary marking, then the same formula still provides an upper bound
  for the length of the curve (\propref{AnyUpper}). Using these two
  propositions, we prove \thmref{IntroMarking} and \thmref{IntroProjection}
  in \secref{Fellow}. These propositions also have analogues in extremal
  length, which we use to sketch an alternate proof of Minsky's theorem at
  the end of \secref{Fellow}. We end the paper with a proof of
  \thmref{IntroFellow} and \thmref{IntroContracting} in \secref{Stable}.
   
  \subsection*{Analogues with Weil-Petersson geodesics} 
  
  As we have mentioned before, a \Teich geodesic is cobounded if and only
  if its endpoints have bounded combinatorics. In \cite{minsky:AWII},
  Brock, Masur and Minsky showed a similar result for bi-infinite geodesics
  in the Weil-Petersson metric on \Teich space. As in our paper, the main
  tool is to show that some projection map is contracting. In their case,
  what they need (and what they show) is that the projection in the pants
  decomposition complex to any hierarchy path satisfying the
  \emph{non-annular bounded combinatorics} property is \emph{coarsely
  contracting} (\cite[Theorem 4.1]{minsky:AWII}).  

  \subsection*{Analogues with Outer space} 
  
  Let $\calX_n$ be the Outer Space, the space of marked metric graphs of
  rank $n$ modulo homothety. The space $\calX_n$ is naturally equipped with
  the Lipschitz metric, on which $\Out(\FF_n) = \Aut(\FF_n) / \Inn(\FF_n)$
  acts as isometries.

  In \cite{algom:SC}, Algom-Kfir established a version of
  \thmref{IntroContracting} for a family of geodesics in $\calX_n$. It was
  shown that the closest-point projection to axes of fully irreducible
  elements of $\Out(\FF_n)$ is strongly contracting. This result gives
  another parallel between fully irreducible elements of $\Out(\FF_n)$ and
  pseudo-Anosov elements of the mapping class group of $S$. A
  generalization of this result for a larger class of paths (lines of
  minima) appears in \cite{ham:LM}. 

  An analogue of \thmref{IntroMarking} exists for $\calX_n$. By a result of
  White, to compute the Lipschitz distance from one graph to another, it
  suffices to consider the length ratios of a finite collection of loops.
  (See \cite[Proposition 2.3]{algom:SC} for a proof of this fact.) 
  
  \subsection*{Acknowledgments}

  We thank the referee for many helpful comments.

\section{Preliminaries} 
 
  \label{Sec:Background}

  \subsection*{\Teich and Lipschitz metrics}
  
  Let \s be a connected, oriented surface of finite type with $\chi(\s)<0$.
  The \Teich space \T of \s is the space of marked conformal structures on \s
  up to isotopy. Via uniformization, \T is also the space of marked
  (finite-area) hyperbolic metrics on \s up to isotopy. 
  
  In this paper, we consider two metrics on \T, the \Teich metric and
  Lipschitz metric. Given  $x, y \in \T$, the \emph{\Teich distance}
  between them is defined to be \[ \dT(x,y) = \frac{1}{2} \inf_f \log
  K(f),\] where $f \from x \to y$ is a $K(f)$--quasi-conformal map
  preserving the marking. (See \cite{gardiner:QT} and \cite{hubbard:TT} for
  background information.) Introduced by Thurston in \cite{thurston:MSM},
  the \emph{Lipschitz distance from $x$ to $y$} is defined to be \[
  \dL(x,y) = \inf_f \log L(f), \] where $f \from x \to y$ is a
  $L(f)$--Lipschitz map preserving the marking. Unlike the \Teich metric,
  the Lipschitz metric is not symmetric, so the order of the two points
  matters when computing distance. 
  
  Both metrics can be described in terms of certain length ratios of
  curves. By a \emph{curve} on $S$, we will always mean a free isotopy
  class of an essential simple closed curve. Essential means the curve is
  not homotopic to a point or a puncture of \s. Given a curve $\alpha$ on
  \s, the extremal length of $\alpha$ in $x \in \T$ is \[\Ext_x(\alpha) =
  \sup_{\rho} \frac{\ell_\rho(\alpha)^2}{\Area(\rho)}, \] where $\rho$ is
  any metric in the conformal class of $x$, $\ell_\rho(\alpha)$ is the
  $\rho$--length of the shortest curve in the homotopy class of $\alpha$,
  and $\Area(\rho)$ is the area of $x$ equipped with the metric $\rho$. For
  the \Teich metric, Kerckhoff showed: 
  \begin{equation} \label{Eq:Kerck} 
     \dT(x,y) = \frac{1}{2} \log\underset{\alpha}\sup
     \frac{\Ext_y(\alpha)}{\Ext_x(\alpha)}, 
  \end{equation} 
  where the $\sup$ is taken over all curves on \s \cite{kerckhoff:AG}.
  For the Lipschitz metric, Thurston showed:
  \begin{equation} \label{Eq:Thurston} 
     \dL(x,y) = \log\underset{\alpha}\sup
     \frac{\ell_y(\alpha)}{\ell_x(\alpha)},
  \end{equation} 
  where $\ell_x(\alpha)$ is the hyperbolic length of $\alpha$ in the unique
  hyperbolic metric in the conformal class of $x$ \cite{thurston:MSM}. 
  
  A point $x \in \T$ is called \emph{$\ep$--thick} (or \emph{$\ep$--thin})
  if the length of the shortest curve on $x$ is greater or equal to $\ep$
  (or less than $\ep$). In the thick part of \T, it is known that the two
  metrics are the same up to an additive error.

  \begin{theorem}\cite{rafi:TL} \label{Thm:TL}
     For every $\ep$ there exists a constant $c$ such that whenever $x, y
     \in \T$ are $\ep$--thick,
     \[ \big| \dT(x,y) - \dL(x,y) \big| \le c.
     \]

  \end{theorem}

  \subsection*{Curve graphs and subsurface projection}
  
  Given two curves $\alpha$ and $\beta$ on \s, we define their
  \emph{intersection number} $\I(\alpha,\beta)$ to be the minimal number of
  intersections between any representatives of homotopy classes of $\alpha$
  and $\beta$. 
  
  The \emph{curve graph} \cc of \s is defined as follows: the vertices are
  curves on \s and the edges are pairs of distinct curves that have minimal
  possible intersections. This minimum is $1$ for the once-punctured torus,
  $2$ for the four-holed sphere, and $0$ for all other surfaces. Note that
  a pair of pants (three-holed sphere) does not have any essential curves.
  We equip \cc with a metric by assigning length one to every edge. 
  
  We use a different definition for the curve graph $\calC(A)$ of an
  annulus $A$ (sphere with two boundary components). By an \emph{arc} on
  $A$ we always mean a homotopy class of a simple arc $\omega$ connecting
  the two boundary components of $A$ where the homotopy is taken relative
  to the endpoints of $\omega$. The intersection $\I(\omega,\omega')$ of
  two arcs is the minimal number of intersections between any
  representatives of homotopy classes of $\omega$ and $\omega'$. The
  vertices of $\calC(A)$ are arcs on $A$ and the edges are pairs of arcs
  with zero intersection. We also equip $\calC(A)$ with a metric as above. 

  From \cite{minsky:CCII}, we recall the definition of \emph{subsurface
  projection} \[ \pi_Y : \cc \to \calP \big( \calC(Y) \big). \] First
  suppose $Y$ is not an annulus. Let $\alpha \in \cc$. If $\alpha$ is
  disjoint from $Y$, then $\pi_Y(\alpha) = \emptyset$ and if $\alpha$ is
  contained in $Y$, then $\pi_Y(\alpha)=\alpha$. In all other cases, the
  restriction of $\alpha$ to $Y$ is a collection of arcs. Let $\omega$ be one such arc. 
  The endpoints of $\omega$ lie on two (not necessarily distinct) boundary
  components $\beta$ and $\beta'$ of $Y$. Let $\calN_\omega$ be a regular
  neighborhood in $Y$ of $\omega \cup \beta \cup \beta'$. Then 
  $\calN_\omega$ always has a boundary component that is a non-trivial 
  curve in $Y$. Let $\pi_Y(\alpha)$ be the union of all essential boundary
  curves of  $\calN_\omega$, 
  where $\omega$ ranges over all arcs in the restriction of $\alpha$ with
  $Y$. The set $\pi_Y(\alpha)$ is non-empty with diameter at most two in
  \cc \cite{minsky:CCII}.   
  
   Given an annular subsurface $A$ of \s with
  \emph{core curve} $\gamma$, the Gromov compactification
  of the annular cover of \s corresponding to $\gamma \in \pi_1(\s)$ is well-defined and 
  is independent of the choice of the hyperbolic metric on \s.  For any $\alpha \in \cc$, the
  projection $\pi_A(\alpha)$ is defined to be the set of lifts of $\alpha$
  to $A$ that are essential arcs. Note that a lift has well-defined endpoints in
  the Gromov boundary of $A$. The set $\pi_A(\alpha)$ has diameter at most
  two in $\calC(A)$.

  \subsection*{Short markings and bounded combinatorics}

  A pants curve system on $S$ is a collection of mutually disjoint curves
  which cut $S$ into pairs of pants. A \emph{marking} $\mu$ on $S$ is a
  pants curve system $\calP$ with additionally a set of \emph{transverse}
  curves $\calQ$ satisfying the following properties. We require each curve
  $\alpha \in \calP$ to have a unique transverse curve $\beta \in \calQ$
  that intersects $\alpha$ minimally (once or twice) and is disjoint from
  all other curves in $\calP$. We will often say $\alpha$ and $\beta$ are
  \emph{dual} to each other, and write $\balpha = \beta$ or $\bbeta =
  \alpha$. This notion of a marking was introduced first by Masur and
  Minsky \cite{minsky:CCII}; however their terminology is clean marking.

  Given $x \in \T$, a \emph{short marking} $\mu_x$ on $x$ is a marking
  where the pants curve system is constructed using the algorithm that
  picks the shortest curve on $x$, then the second shortest disjoint from
  the first, and so on. Once the pants curve system is complete, the
  transverse curves are then chosen to be as short as possible. Note that a
  short marking on $x$ may not be unique, but all short markings on $x$
  form a bounded set in \cc. Thus, we will refer to $\mu_x$ as \emph{the
  associated short marking} on $x$.
  
  Let $x, y \in \T$ and $\mu_x$ and $\mu_y$ be the associated short
  markings. For any $Y \subseteq S$, define 
  \[ d_Y(\mu_x, \mu_y) = \diam_{\calC(Y)} \big( \pi_Y(\mu_x), \pi_Y(\mu_y)
  \big), \] where $\pi_Y(\mu_x)$ is the union of the projection of the
  curves of $\mu_x$ to $Y$.
  
  \begin{definition} \label{Def:BC}
    
    Two points $x, y \in \T$ are said to have $K$--\emph{bounded
    combinatorics} if there exists a constant $K > 0 $ such that for every
    proper subsurface $Y \subset S$, 
  \[ 
    d_Y(\mu_x, \mu_y) \le  K. 
  \] 
  
  \end{definition}

  \subsection*{Cobounded geodesics}

  Given $x,y \in \T$, we denote by $\GT(x,y)$, or $\GT$ when endpoints are
  not emphasized, the \Teich geodesic connecting $x$ and $y$. We denote by
  $\GL(x,y)$ (or $\GL$) a Lipschitz geodesic \emph{from $x$ to $y$}. In
  either the \Teich or the Lipschitz metric, a geodesic is
  \emph{$\ep$--cobounded} if every point on the geodesic is $\ep$--thick.
  Given $x$ and $y$, there is a unique \Teich geodesic connecting them. On
  the other hand, Thurston proved the existence of a Lipschitz geodesic
  from $x$ to $y$ \cite{thurston:MSM}, but in general there may be more
  than one. 

  The following theorem is due to Rafi. The second direction also follows
  from the work of Minsky (see \cite{minsky:TG} and \cite{minsky:ELCI}).  

  \begin{theorem}[\cite{rafi:SC}] \label{Thm:Teich}
   
     For every $\ep, K > 0$, there exists a constant $\ep' > 0$ such that
     the following holds. If $x, y \in \T$ are $\ep$--thick and have
     $K$--bounded combinatorics, then the \Teich geodesic $\GT$ with
     endpoints $x$ and $y$ is $\ep'$--cobounded.
    
     Conversely, for every $\ep$ there is $K'$ such that if $\GT$ is
     $\ep$--cobounded (possibly an infinite or bi-infinite ray), then any
     two points on $\GT$ have $K'$--bounded combinatorics.

  \end{theorem}
  
  For the rest of this paper, we will fix $\ep>0$ to be less than the
  Margulis constant. Unless otherwise specified, by thick or thin, we will
  always mean $\ep$--thick or $\ep$--thin. We will also fix a constant $K$
  so that bounded combinatorics will mean $K$--bounded combinatorics. A
  cobounded geodesic will always mean $\ep'$--cobounded with $\ep'$ as in
  \thmref{Teich}. Once $\ep$ and $K$ are fixed, the dependence of other
  constants on $\ep$ and $K$ can be ignored; we can treat constants which
  depend on $\ep$ and $K$ as if they depended only on the topology of $S$.
  
  In this paper, we will try to understand a Lipschitz geodesic $\GL$ whose
  endpoints have bounded combinatorics. Our main tool will be to compare
  the geometry of $\GL$ with the geometry of the unique \Teich geodesic
  $\GT$ connecting the same endpoints. We will use the fact that $\GT$ is
  cobounded to show that the closest-point projection to $\GT$ in the
  Lipschitz metric is contracting (\thmref{Projection}). This will imply
  that $\GL$ and $\GT$ fellow travel, and hence $\GL$ is also cobounded
  (for some $\ep''$ depending only on \s) (\thmref{Fellow}). 

  \subsection*{Thick-thin decomposition of a hyperbolic surface}

  Fix $0 < \ep_1 < \ep_0 < \ep$. For any $x \in \T$, we recall the
  notion of $(\ep_0, \ep_1)$ \emph{thick-thin decomposition} of $x$ (see
  \cite{minsky:PR}). Let $\calA$ be the (possibly empty) set of curves in
  $x$ whose hyperbolic lengths are less than $\ep_1$. For each $\alpha \in
  \calA$, let $A_{\alpha}$ be the regular neighborhood of the $x$--geodesic
  representative of $\alpha$ with boundary length $\ep_0$. Note that, since
  $\epsilon_0$ is less than the Margulis constant, the annuli are disjoint.
  Let $\calY$ be the set of components of $x \setminus (\bigcup_{\alpha \in
  \calA} A_{\alpha}).$ We denote this decomposition of $x$ by $(\calA,
  \calY)$.
  
  Note that if $(\calA,\calY)$ is a thick-thin decomposition for $x$ and
  $\mu_x$ is a short marking, then $\calA$ always forms a subset of the
  pants curve system in $\mu_x$. 
  
  \subsection*{Notations}

  Throughout this paper we will adopt the following notations. Below, $\ga$
  and $\gb$ represent various quantities such as distances between two
  points or lengths of a curve, and $C$ and $D$ are constants that depend
  only on the topology of $S$. 
  
  \begin{enumerate}
    
     \item $\ga \lmul \gb \quad$ if $\quad \ga \le C \gb $,

     \item $\ga \ladd \gb \quad$ if $\quad \ga \le \gb + D$,
     
     \item $\ga \emul \gb \quad$ if $\quad \ga \lmul \gb$ and $\gb \lmul \ga$.
     
     \item $\ga \eadd \gb \quad$ if $\quad \ga \ladd \gb$ and $\gb \ladd \ga$.

  \end{enumerate}
  We will also often use the notation $\ga = O(1)$ to mean $\ga \lmul 1$.

\section{Hyperbolic length estimates via markings}

  \label{Sec:Marking}

  In this section we give some estimates of the hyperbolic length of a
  simple closed curve in terms of the number of times the curve intersects
  a marking on a surface and the length of the marking itself. Up to a
  multiplicative error, our expression provides an accurate estimate when
  the marking is short, but yields only an upper bound for a general
  marking. 
  
  \subsection*{Short Marking} \label{Sec:ShortMarking}
      
  \begin{proposition}\label{Prop:ShortEquality} 
      
     Let $x \in \T$ and $\mu_x$ be a short marking on $x$. Then for every
     curve $\gamma$, 
     \[ 
       \ell_x(\gamma) 
       \emul 
       \sum_{\alpha \in \mu_x} \I(\gamma, \alpha) \, \ell_x(\balpha),
     \]
     and
     \[
       \Ext_x(\gamma) 
       \emul 
       \sum_{\alpha \in \mu_x} \I(\gamma, \alpha)^2 \, \Ext_x(\balpha).  
     \] 
  
  \end{proposition}

  \begin{proof}

     We first prove the statement for the hyperbolic length of $\gamma$.
     Consider the $(\ep_0,\ep_1)$-decomposition  $(\calA, \calY)$ for $x$.
     For each $Y\in \calY$, let  $\mu_Y$ be the set of curves in $\mu_x$
     that are contained entirely in $Y$. Note that if $\alpha \in \mu_Y$,
     then so is $\balpha$. The set $\mu_Y$ fills the surface $Y$, that is,
     every curve in $Y$ intersects some curve in $\mu_Y$. For every curve
     $\gamma$ in $Y$ define
     \[
        \I(\gamma, \mu_Y) = \sum_{\alpha \in \mu_Y} \I(\gamma,\alpha).
     \] 
     It is a consequence of \cite[Corollary 3.2]{rafi:LT} and
     \cite{minsky:TG} that $\ell_x(\gamma)$ can be estimated using the
     following sum: 
     \begin{equation} \label{Eq:HypMinsky} 
        \ell_x(\gamma) 
        \emul \sum_{Y\in\calY} \I(\gamma,\mu_Y) + 
        \sum_{\alpha\in\calA} \I(\gamma,\alpha)\left[
        \log{\frac{1}{\ell_x(\alpha)}} + \ell_x(\alpha)
        \twist_\alpha(x,\gamma)\right]. 
     \end{equation} 
     Here, $\twist_\alpha(x,\gamma) = d_A(\balpha, \gamma)$ (see
     \cite{minsky:PR} and \cite{rafi:HT} for more details). We need to show
     \begin{equation}\label{Eq:HypEstimate}
        \ell_x(\gamma) \emul 
        \sum_{Y \in \calY} \sum_{\alpha\in\mu_Y} 
        \I(\gamma, \alpha) \, \ell_x(\balpha) + 
        \sum_{\alpha\in\calA}  \Big[ \I(\gamma,\alpha) \,
        \ell_x(\balpha)
        + \I(\gamma,\balpha) \, \ell_x(\alpha) \Big] 
     \end{equation} 
     which is just a rephrasing of the statement of the proposition for the
     hyperbolic length. We will show that the right hand sides of Equations
     \eqref{Eq:HypMinsky} and \eqref{Eq:HypEstimate} are comparable. 
     
     To start, note that for every $\alpha \in \mu_Y$, we have
     $\ell_x(\balpha) \emul 1$. Hence
     \begin{equation} \label{Eq:Thick}
        \sum_{Y \in \calY} \sum_{\alpha\in\mu_Y}
        \I(\gamma, \alpha)\, \ell_x(\balpha) \emul
        \underset{Y\in\calY}{\sum} \I(\gamma,\mu_Y).
     \end{equation}
     Now consider $\alpha \in\calA$. 
     By the collar lemma, the hyperbolic length of the dual curve 
     $\balpha$ is roughly the width of the collar around $\alpha$. That is, 
     \[
        \ell_x(\balpha)  \emul  \log \frac{1}{\ell_x(\alpha)}.
     \]
     Summing over $\alpha \in \calA$ we obtain
     \begin{equation} \label{Eq:Collar} 
       \sum_{\alpha \in \calA } \I(\gamma, \alpha) \, \ell_x(\balpha) \emul
       \sum_{\alpha \in \calA } \I(\gamma, \alpha) \, \log
       \frac{1}{\ell_x(\alpha)}.
     \end{equation}
     We now compare the last terms. Assume $\gamma$ intersects
     some curve $\alpha\in \calA$. From the discussion in
     \cite[Section 3]{minsky:PR} we have
     \[
       \twist_\alpha(x,\gamma) \ladd  \frac{\I(\gamma,\balpha)}{\I(\gamma,
       \alpha)}. 
     \]
     To make the error multiplicative, we add a large term to the right side:
     \[
        \twist_\alpha(x,\gamma) \lmul
        \frac{\ell_x(\balpha)}{\ell_x(\alpha)} +
        \frac{\I(\gamma,\balpha)}{\I(\gamma, \alpha)}.
     \]
     Summing over $\alpha \in \calA$ and multiplying by 
     $\I(\gamma, \alpha) \, \ell_x(\alpha)$ we obtain
     \[  
        \sum_{\alpha \in \calA } \I(\gamma, \alpha)\, \ell_x(\alpha)
        \twist_\alpha(x,\gamma) 
        \lmul \sum_{\alpha \in \calA } \I(\gamma, \alpha) \, \ell_x(\balpha) +
         \I(\gamma, \balpha) \,\ell_x(\alpha).
     \]
     Thus the right hand side of \eqref{Eq:HypMinsky} is bounded above by
     the right hand side of \eqref{Eq:HypEstimate} up to a multiplicative
     error. 
     
     It remains to find an upper bound for $\I(\gamma, \balpha)\,
     \ell_x(\alpha) $, $\alpha \in \calA$, using terms in the right hand
     side of \eqnref{HypMinsky}. Since our inequalities are up to a
     multiplicative error, finding an upper bound for each such term
     provides an upper bound for the sum.
     
     Consider the regular neighborhood $A_\alpha$ of $\alpha$. If $\ep_0$
     is small enough, $\gamma$ intersects $\alpha$ every time it enters
     $A_\alpha$. The number of intersection points between $\gamma$ and
     $\balpha$ inside of $A_\alpha$ is bounded by
     $\I(\gamma,\alpha)\twist_\alpha(x,\gamma) $ and the number of
     intersection points outside of $A_\alpha$ is less than the number of
     intersection points between $\gamma$ and $\calP$, the set of pants
     curves in $\mu_x$ (every time $\gamma$ intersects $\balpha$ it either
     twists around $\alpha$ and intersects $\alpha$ or it will intersect
     some curve in $\calP$ before intersecting $\balpha$ again). That is, 
     \[ 
       \I(\gamma, \balpha) \lmul \I(\gamma,\alpha)\, \twist_\alpha(x,\gamma
        ) + \I(\gamma, \calP). 
     \]
     Since, for any $\beta \in \calP$, $\ell_x(\alpha) \leq \ell_x(\bbeta)$
     we have
     \[
        \I(\gamma, \balpha) \, \ell_x(\alpha)\lmul \I(\gamma,\alpha)
        \ell_x(\alpha) \, \twist_\alpha(x, \gamma) + \sum_{\beta \in
        \calP}\I(\gamma, \beta) \ell_x(\bbeta). 
    \]
    Up to a multiplicative error, this is less than the right hand side of
    \eqref{Eq:HypMinsky}. Thus the right hand side of
    \eqref{Eq:HypEstimate} is bounded above by the right hand side of
    \eqref{Eq:HypMinsky} up to a multiplicative error. Therefore, the two
    quantities are equal. This completes the proof of the first statement. 
     
     To prove the statement for extremal length, we can follow the same
     path. We have the following estimate for the extremal length of
     a curve (this is Theorem 7 in \cite{rafi:LQC} which follows essentially 
     from \cite{minsky:PR}) analogous to \eqnref{HypMinsky}: 
     \begin{equation} \label{Eq:ExtMinsky} 
        \Ext_x(\gamma) \emul 
        \underset{Y\in\calY}{\sum} \I(\gamma,\mu_Y)^2  + 
        \underset{\alpha\in\calA}{\sum} \I(\gamma,\alpha)^2\left[ \frac{1}{\Ext_x(\alpha)}
        + \Ext_x(\alpha) \twist_\alpha (x,\gamma)^2 \right]
         \quad \nonumber 
     \end{equation} 
     Similar to \eqnref{Thick}, we have
     \begin{equation}\label{Eq:ExtThick}
       \underset{Y\in\calY}{\sum} \I(\gamma,\mu_Y)^2 
       \emul 
       \sum_{\alpha\in\mu_Y}
       \I(\gamma, \alpha)^2 \, \Ext_x(\balpha). \nonumber 
     \end{equation} 
     For any $\alpha \in \calA$, the version of the collar lemma for
     extremal length says:
     \[
       \Ext_x(\balpha)
        \emul 
       \frac{1}{\Ext_x(\alpha)}.
     \] 
     The rest of the proof is essentially identical.  
  \end{proof} 

  \subsection*{Upper bound from any marking}
  
  \label{Sec:AnyMarking}
  
  In the following, we use a surgery argument on curves to derive an upper
  bound for the hyperbolic length of a curve using an arbitrary marking.
  Although we do not need such a precise estimate, our argument produces a
  multiplicative error of 2. 
  
  \begin{proposition} \label{Prop:AnyUpper}
  
     Let $x \in \T$ and $\mu$ be an arbitrary marking on $S$. Then for
     every curve $\gamma$, 
     \begin{equation} \label{Eq:UpperBound} 
        \ell_x(\gamma) 
        \lmul 
        \sum_{\alpha \in \mu} \I(\gamma, \alpha) \, \ell_x(\balpha) 
     \end{equation}

  \end{proposition}

  The outline of the proof is as follows. Let $\calP$ be the pants curve
  system in $\mu$. We first perturb $\gamma$ so that the restriction of
  $\gamma$ to every pair of pants $P \in S \setminus \calP$ is a union of
  \emph{admissible} arcs. These are arcs for which the inequality
  \eqref{Eq:UpperBound} holds. Perturbing $\gamma$ will only increase its
  length. Hence, if \eqref{Eq:UpperBound} holds for every arc, it holds
  for $\gamma$ as well. 

  \subsection*{Admissible arcs} 
  
  Let $P$ be a pair of (embedded) pants in the pants decomposition
  associated with the marking $\mu$. Equip $P$ with the hyperbolic metric
  inherited from $x$. For every boundary curve $\alpha \in \partial P$, let
  $\balpha$ be a simple geodesic arc in $P$ with endpoints on $\alpha$
  separating the other two boundary components of $P$, and let $E$ be the
  set of endpoints of arcs $\balpha$. Let $\omega$ be any simple geodesic
  arc whose endpoints are in $E$, and let $\I(\omega, \balpha)$ represent
  the number of intersection points in the interior of $P$. Assume that one
  endpoint of $\omega$ lies in $\alpha_-$ and the other lies in $\alpha_+$.
  We say $\omega$ is \emph{admissible} if 
  \[
     \ell_x(\omega) \lmul 
     \ell_x(\balpha_+) + \ell_x(\balpha_-) + \I(\omega, \balpha_+) \,
     \ell_x(\alpha_+) + \I(\omega, \balpha_-) \, \ell_x(\alpha_-). 
  \]
  As we shall see, most arcs are admissible. 

  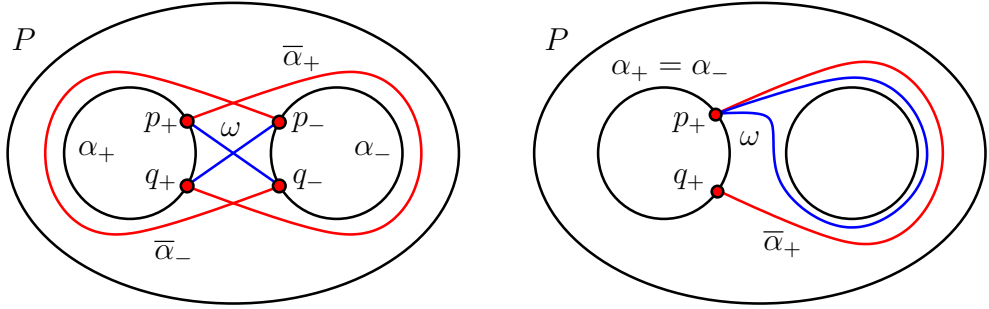
\begin{figure}[ht]
  \setlength{\unitlength}{3.6pt}
  \begin{picture}(104,34)
     \put(0,0){\includegraphics{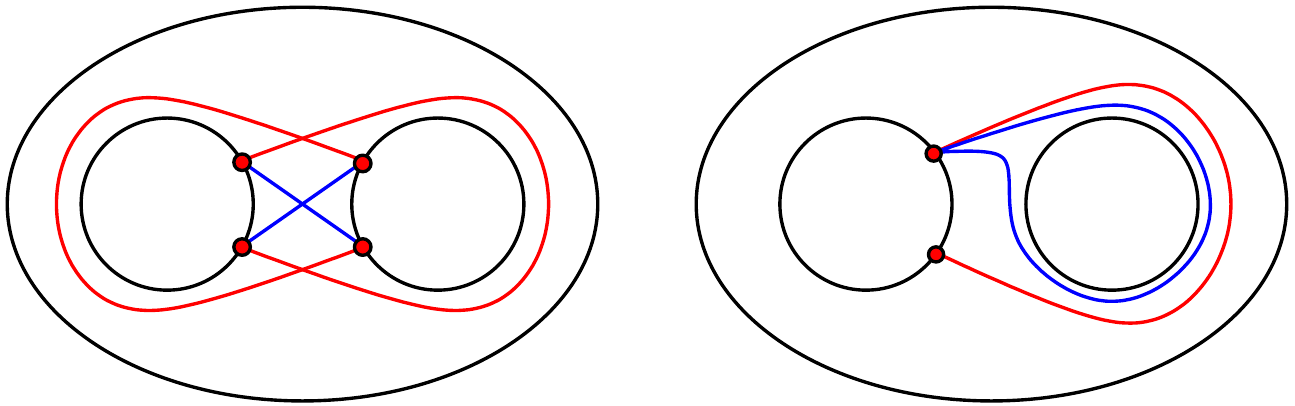}}
     \put(37,16){$\alpha_-$}
     \put(16,5){$\balpha_-$}
     \put(8,16){$\alpha_+$}
     \put(29.5,25.5){$\balpha_+$}
     \put(15,19){$p_+$}
     \put(15,13){$q_+$}
     \put(30.5,19){$p_-$}
     \put(30.5,13){$q_-$}
     \put(64,24.5){$\alpha_+=\alpha_-$}
     \put(80,6){$\balpha_+$}
     \put(70.5,19){$p_+$}
     \put(70.5,13){$q_+$}
     \put(1,27){$P$}      
     \put(57,27){$P$}  
     \put(23,18){$\omega$}  
     \put(77.5,17){$\omega$}  
  \end{picture}
  \caption{There are 12 non-admissible arcs in $P$. For each pair of
  distinct boundary components of $P$, there are two non-admissible arcs as
  depicted in the left figure (both arcs are labeled $\omega$). For each
  boundary component of $P$, there are two non-admissible arcs. The figure
  on the right depicts one such arc $\omega$. The second one is obtained
  via a reflection across the $x$-axis. } \label{Fig:Not-Pants}
  \end{figure}

  \begin{lemma} \label{Lem:Admissible-Pants}
    
    Let $\omega$ be a simple geodesic arc with endpoints in $E$. 
    Then $\omega$ is admissible unless it is one of the arcs depicted
    in \figref{Not-Pants}. In particular, if
    $\I(\omega, \balpha)>0$ for some $\alpha \in \partial P$ then 
    $\omega$ is admissible. 
    
  \end{lemma}
  
  \begin{proof}
    
    First suppose $\omega$ starts and ends on two different boundary
    components of $P$. Let $\omega_1$ and $\omega_2$ be the arcs depicted
    in \figref{Admissible-Pants}. Then, up to homotopy, $\omega$ is a
    concatenation of either $\omega_1$ or $\omega_2$ with several copies of
    $\alpha_+$, several copies of $\alpha_-$ and at most one copy of the
    arcs $[p_+, q_+]$ or $[p_-, q_-]$. The number of copies of $\alpha_+$
    needed is at most $\I(\omega, \balpha_+)$ and the number of copies of
    $\alpha_-$ needed is at most $\I(\omega, \balpha_-)$. The length of
    $\omega$ is less than the sum of these arcs. 

    \begin{figure}[ht] \setlength{\unitlength}{3.6pt}
       \begin{picture}(58,40)
          \put(0,0){\includegraphics{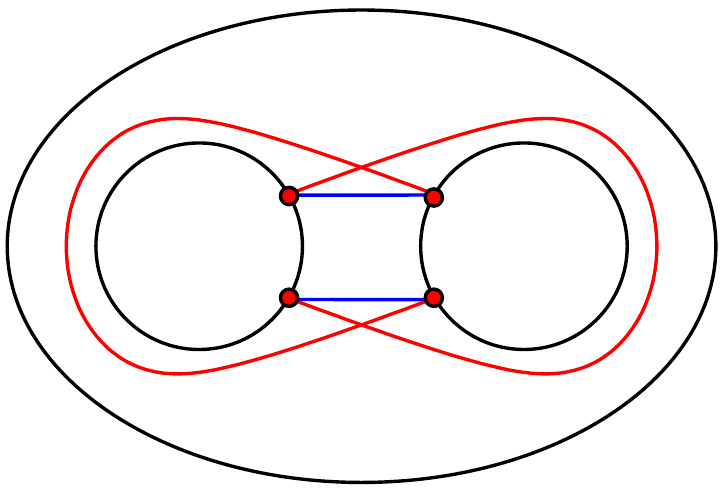}}
          \put(45,19){$\alpha_-$}
          \put(9,19){$\alpha_+$}
          \put(16,6){$\balpha_-$}
          \put(39,31){$\balpha_+$}
          \put(18.5,23){$p_+$}
          \put(18.5,15){$q_+$}
          \put(36.5,23){$p_-$}
          \put(36.5,15){$q_-$}
          \put(27,21.5){$\omega_1$}
          \put(27,16.5){$\omega_2$}            
          \put(2,32){$P$}      
       \end{picture}
       \caption{The arcs $\omega_1$ and $\omega_2$.} 
       \label{Fig:Admissible-Pants}
    \end{figure}
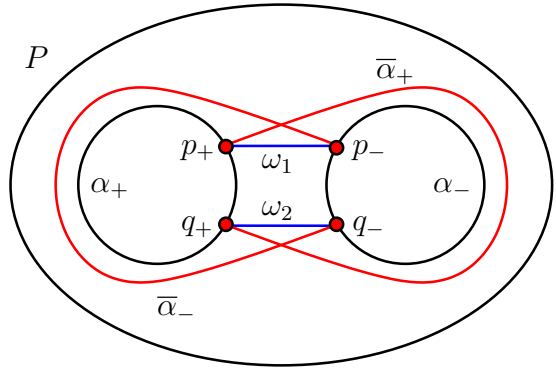
    
    Note that the lengths of $\omega_1$ and $\omega_2$ are both less than $
    \ell_x(\balpha_+) + \ell_x(\balpha_-)$. The lengths of copies of
    $\alpha_\pm$ needed is less than or equal to  $\I(\omega,
    \balpha_\pm)\,\ell_x(\alpha_\pm)$. If either $\I(\omega, \balpha_+)$ or
    $\I(\omega, \balpha_-)$ is non-zero then the quantity $\I(\omega,
    \balpha_\pm)\,\ell_x(\alpha_\pm)$ is also an upper bound for the length
    of the segment $[p_\pm, q_\pm]$. Hence, if $\omega$ is not admissible,
    then it is disjoint from $\balpha_\pm$ and it is not $\omega_1$ or
    $\omega_2$. The arcs depicted in the left side of \figref{Not-Pants}
    are the only possibilities. 
    
    A similar argument works when $\omega$ starts and ends on the same
    curve, that is, when $\alpha_+ = \alpha_-$. In this case, if $\omega$
    is not admissible, then it must be disjoint from $\balpha_+$ but not
    equal to it. There are only two such arcs, one with both endpoints at
    $p_+$ (see the right side of \figref{Not-Pants}) and one with both
    endpoints at $p_-$. 
  \end{proof}
  
  In the case that a pair of pants is not embedded in $x$ (when one curve
  in $x$ appears twice as a boundary of a pair of pants), the dual curve
  does not intersect the pants curves twice and the above arguments do not
  apply. In this case, the definition of an admissible arc has to be
  modified. Let $T$ be a torus with one boundary component that is an image
  of a pair of pants associated to $\mu$. Let $\alpha$ be the boundary
  curve of $T$ and $\balpha$ be a simple geodesic arc with endpoints on
  $\alpha$. Also, let $\beta$ be a simple closed curve in $T$ that is
  disjoint from $\balpha$, and let $\bbeta$ be the dual curve to $\beta$: a
  simple closed geodesic that intersects each of $\beta$ and $\balpha$
  exactly once. Let $E=\{p,q\}$ be the endpoints of $\balpha$, and let
  $\omega$ be a simple geodesic arc with endpoints in $E$. We say $\omega$
  is admissible if 
  \[
     \ell_x(\omega) \lmul \ell_x(\balpha) 
     + \I(\omega, \balpha)\,\ell_x(\alpha) 
     + \I(\omega, \beta)\,\ell_x(\bbeta)
     + \I(\omega, \bbeta)\,\ell_x(\beta).  
  \]
  
  \begin{lemma} \label{Lem:Admissible-Torus}
    
    Let $\omega$ be a simple geodesic arc with endpoints in $E$. Then
    $\omega$ is admissible unless it is an arc of a type depicted in
    \figref{Not-Torus}. In particular, if $\I(\omega, \balpha)>0$ then
    $\omega$ is admissible. 
    
  \end{lemma}

  \begin{figure}[ht] \setlength{\unitlength}{3.6pt}
     \begin{picture}(88,38)
        \put(0,0){\includegraphics{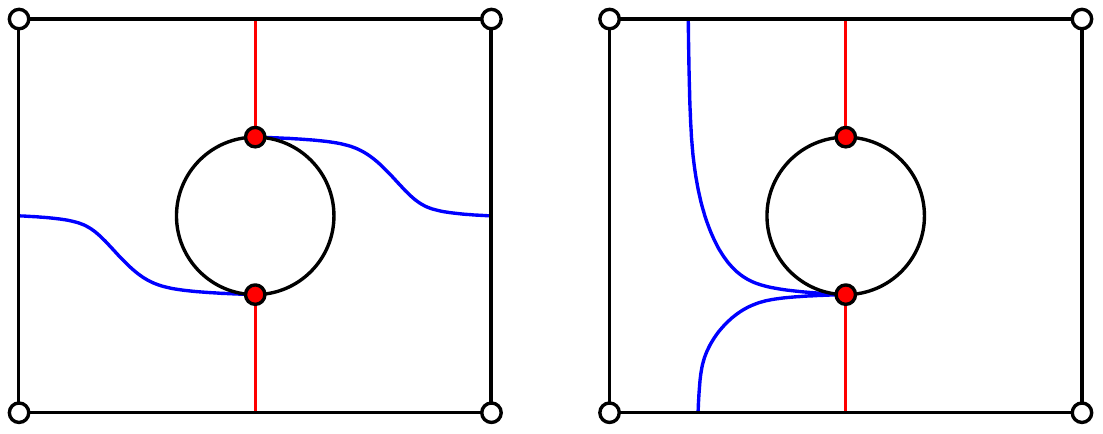}}
        \put(-2,22){$\beta$}
        \put(10,29){$\bbeta$}
        \put(19.5,19.8){$p$}
        \put(19.5,13.8){$q$}
        \put(22,4.5){$\balpha$}
        \put(26,11.5){$\alpha$}
        \put(5,13.5){$\omega$}
        \put(28,35){$T$}      
        \put(88.5,22){$\beta$}
        \put(76.5,29){$\bbeta$}
        \put(67,19.8){$p$}
        \put(67,13.8){$q$}
        \put(69.5,4.5){$\balpha$}
        \put(73.5,11.5){$\alpha$}
        \put(54,13.5){$\omega$}
        \put(58.5,35){$T$}      
     \end{picture}
     \caption{There are 6 non-admissible arcs in $T$. On the left is a
     non-admissible arc $\omega$ whose endpoints are distinct. Another
     non-admissible arc of the same type can be obtained via a reflection
     across the $x$--axis. On the right is a non-admissible arc $\omega$
     which starts and ends at the same point. The other $3$ non-admissible
     arcs of this type can be obtained via reflections across the $x$--axis
     and the $y$-axis.} \label{Fig:Not-Torus}
  \end{figure}
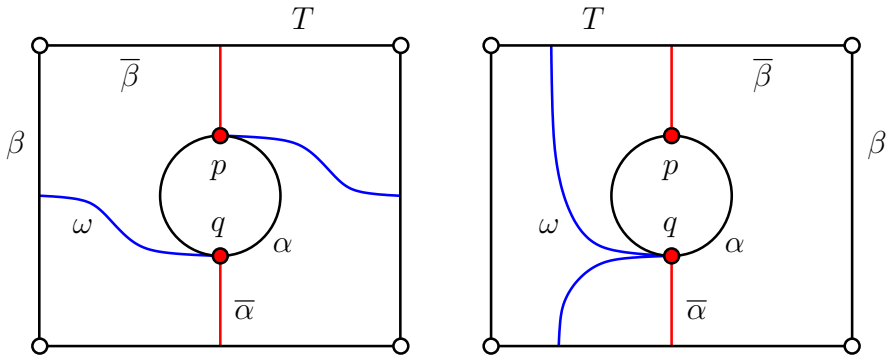    

  \begin{proof}  
  
     Up to homotopy, the arc $\omega$ is a concatenation of several copies
     of $\alpha$, one-half of $\balpha$, a simple closed curve $\delta$,
     then again one-half of $\balpha$ (could be the same half or the other
     half), and finally several copies of $\alpha$. One may have to add the
     arc $[p,q]$ to the beginning or to the end to ensure the arc described
     above and $\omega$ have the same endpoints. First we claim 
     \[
       \ell_x(\delta) \leq \I(\omega, \beta)\,\ell_x(\bbeta)+ \I(\omega,
       \bbeta)\,\ell_x(\beta).  
     \] 
     Consider the fundamental group of $T$ with a base point at the
     intersection of $\beta$ and $\bbeta$. Then a curve homotopic to
     $\delta$ can be written as a product of copies of $\beta$ and $\bbeta$. The
     number of copies of $\beta$ and $\bbeta$ needed is exactly $\I(\omega,
     \bbeta)$ and $\I(\omega, \beta)$ respectively. This proves the claim. 

     The number of copies of $\alpha$ needed is bounded above by
     $\I(\omega, \balpha)$. If $\I(\omega, \balpha)$ is non-zero then the
     quantity $\I(\omega, \balpha)\,\ell_x(\alpha)$ is also an upper bound
     for the length of the segment $[p, q]$.  Hence, $\omega$ is admissible
     if $\I(\omega, \balpha)>0$ or if the arc $[p,q]$ is not required to
     construct $\omega$. Arcs  of type depicted in \figref{Not-Torus} are
     the only exceptions.  
  \end{proof}

  \begin{proof}[Proof of \propref{AnyUpper}]
    
     If $\gamma$ is a curve in $\mu$ then the statement of the proposition
     is clearly true. We can further assume that there is a pants curve
     $\alpha_0 \in \mu$ so that $\gamma$ intersects both $\alpha_0$ and
     $\balpha_0$. Otherwise, $\gamma$ has to pass only through pants in the
     form discussed in \lemref{Admissible-Torus}. That means, $S$ is a
     union of two one-holed tori. That is, $S$ is a genus two surface and
     $\mu$ and $\gamma$ are as depicted in \figref{Exception}. In this
     case, it is easy to produce a curve homotopic to $\gamma$ as a
     concatenation of curves in $\mu$ and hence the proposition holds. 
    
     \begin{figure}[ht] \setlength{\unitlength}{3.6pt}
     \begin{picture}(64,34)
        \put(0,0){\includegraphics{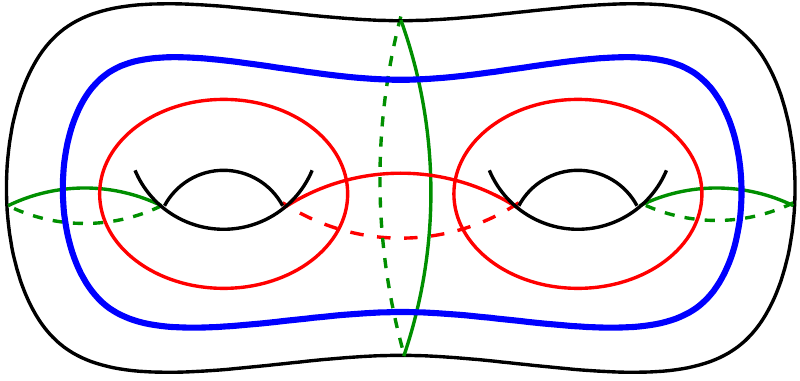}}   
        \put(8,26.5){$\gamma$}  
        \put(-3.5,13.5){$\alpha_1$}  
        \put(31,-1){$\alpha_2$}  
        \put(64.5,13.5){$\alpha_3$}  
     \end{picture}
     \caption{The thick curve which goes around both holes of the surface
     is $\gamma$. The union of the other curves form the marking $\mu$. The
     curves $\alpha_1$, $\alpha_2$, and $\alpha_3$ are the pants curves of
     $\mu$. For each $i$, the transverse curve $\balpha_i$ to $\alpha_i$ is
     the unlabeled curve which intersects only $\alpha_i$. The curve
     $\gamma$ does not intersect both $\alpha_i$ and $\balpha_i$ for any
     $i=1,2,3$.} \label{Fig:Exception}
     \end{figure}
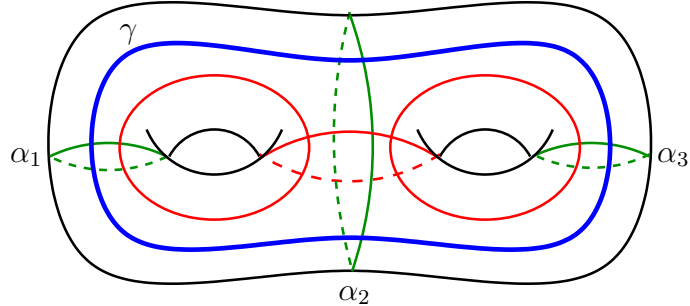      
    
     We claim $\gamma$ can be homotoped to a curve $\gamma'$ so that
     $\gamma'$ is a union of admissible arcs and a sub-arc of $\alpha_0$.
     The curve $\gamma'$ has the same intersection pattern with the pants
     curves of $\mu$ and the intersection number of $\gamma$ with every
     transverse curve is the same as the sum of the \emph{interior
     intersection number} of $\gamma'$ with these curves. The proposition
     then follows from \lemref{Admissible-Pants} and
     \lemref{Admissible-Torus}.
    
     First perturb $\gamma$ slightly so that it does not pass through any
     intersection point between $\alpha$ and $\balpha$ for a pants curve
     $\alpha \in \mu$. We change $\gamma$ by replacing the restriction of
     $\gamma$ to a pair of pants $P$ or a torus $T$ to admissible arcs.
     Start with the pair of pants $P_0$ with the boundary curve $\alpha_0$
     and a sub-arc $\omega_0$ of $\gamma$ that starts from $\alpha_0$ and
     ends in $\alpha_1$ ($\alpha_1$ may equal $\alpha_0$). Replace
     $\omega_0$ with an admissible arc $\omega_0'$ that has the same
     intersection pattern with the dual arcs in $P_0$. Let $r_0$ and $r_1$
     be the endpoints of $\omega_0'$ in $\alpha_0$ and $\alpha_1$
     respectively. Now let $P_2$ be the pair of pants (or once-punctured
     torus) with  $\alpha_1$ as a boundary component that is not $P_0$ and
     let $\omega_1$ be the continuation of $\omega_0$ in $P_1$. Again,
     replace $\omega_1$ with an admissible arc $\omega_1'$, but make sure
     $\omega_1'$ starts at $r_1$. This is always possible by
     \lemref{Admissible-Pants} and \lemref{Admissible-Torus}; we can push
     the intersection point of $\omega_1$ with $\alpha_2$ either to the
     right or to the left and one of these two will result in an admissible
     arc. Continue in this fashion, replacing the arc $\omega_k$ which is a
     continuation of $\omega_{k-1}$ in the pair of pants (or once-punctured
     torus) $P_k$ with an admissible arc making sure that the starting
     point $r_k$ of $\omega_k'$ matches the endpoint of $\omega_{k-1}'$. We
     can do this until we reach the starting point after $K$ steps. Then
     $\alpha_K = \alpha_0$. We can ensure the arc $\omega_{K}'$ is
     admissible and it starts from $r_{K-1}$. But $r_K$ may not equal
     $r_0$. In this case, we add a sub-arc $\omega'$ of $\alpha_0$ to close
     up $\gamma'$ to a curve homotopic to $\gamma$. 
    
     If we now add up the inequalities defining admissibility, we get that
     the sum of the lengths of arcs $\omega_i'$ is less than the right-hand
     side of the inequality \eqref{Eq:UpperBound}. Also the term
     $\ell_x(\balpha_0)$ appears in the right hand side of
     \eqref{Eq:UpperBound} and provides an upper bound for the length of
     $\omega'$. That is, the right-hand side of \eqnref{UpperBound} is an
     upper bound for the length of $\gamma'$ and hence for
     $\ell_x(\gamma)$. This finishes the proof.  
  \end{proof}
   
  \begin{remark} \label{Rem:ExtUpper}
       
     If $x$ is in the thick part of \Teich space, then \propref{AnyUpper}
     also holds for extremal length. This follows from the fact that in the
     thick part, hyperbolic length is coarsely equal to the square root of
     the extremal length (see \lemref{ExtHypThick}).
        
  \end{remark}

\section{Bounded Projection to a \Teich geodesic}

  \label{Sec:Fellow}
       
  In this section, our main goal is to prove \thmref{IntroProjection} of
  the introduction. The first step is to prove \thmref{IntroMarking}, which
  allows us to estimate the Lipschitz distance from $x$ to $y$ by
  considering only how much a short marking on $x$ is stretched. The
  special case of \thmref{IntroMarking} when both $x$ and $y$ are in the
  thick part was proved in \cite{rafi:TL}. We restate \thmref{IntroMarking}
  below.
  
  \begin{theorem}[Candidate curves]\label{Thm:ShortMarking}
     
     Let $x,y\in \T$ and let $\mu_x$ be a short marking on $x$. Then 
     \[ 
        \dL(x,y) 
        \eadd
        \log\max_{\alpha \in \mu_x} \frac{\ell_y(\alpha)}{\ell_x(\alpha)}.
        \qedhere
     \]
  A curve $\alpha \in \mu_x$ satisfying $\dL(x,y) \eadd \log
  \frac{\ell_y(\alpha)}{\ell_x(\alpha)}$ is called a \emph{candidate}
  curve from $x$ to $y$.
  
  \end{theorem}

  \begin{proof}[Proof of \thmref{ShortMarking}]
     
     By Thurston's theorem (\eqnref{Thurston}), there exists a curve
     $\gamma$ such that $\log \frac{\ell_y(\gamma)}{\ell_x(\gamma)}$ is
     within a uniform additive error of $\dL(x,y)$. We invoke
     \propref{ShortEquality} and \propref{AnyUpper} to compute the
     hyperbolic length of $\gamma$ on $x$ and $y$, using the fact that
     $\mu_x$ is short on $x$ but may not be short on $y$: 
     \[ 
        \ell_x(\gamma) 
        \emul
        \sum_{\alpha_\in\mu_x} \I(\gamma, \alpha) \ell_x(\balpha), 
        \qquad 
        \ell_y(\gamma) 
        \lmul
        \sum_{\alpha_\in\mu_x} \I(\gamma, \alpha) \ell_y(\balpha).
     \] 
     We have
     \begin{align*} 
     e^{\dL(x,y)} 
       \emul \frac{\ell_y(\gamma)}{\ell_x(\gamma)} 
     & \lmul \frac{\sum_{\alpha_\in\mu_x} \I(\gamma, \alpha)\,
       \ell_y(\balpha)} {\sum_{\alpha_\in\mu_x} 
       \I(\gamma, \alpha) \, \ell_x(\balpha)} \\ 
        &  \lmul \max_{\alpha\in \mu_x}
         \frac{\ell_y(\alpha)}{\ell_x(\alpha)}.
     \end{align*}
     The opposite inequality directly follows from the definition of
     Lipschitz distance.
  \end{proof} 

  Given a closed set $\calK \subset \T$ and
  $x \in \T$, define 
  \[ 
    \dL(x, \calK) = \underset{y \in \calK}{\inf} \, \dL(x,y). 
  \]
  The \emph{closest-point projection} of $x \in \T$ to $\calK$ with
  respect to the Lipschitz metric is 
  \[ 
     \pi_{\calK}(x) = \big\{ y \in \calK \st \dL(x,y) = \dL(x, \calK) \big\}.
  \] 
  The projection is always nonempty, but it could contain more than one
  point. We can also project a set $B \subset \T$ to $\calK$:
  $\pi_{\calK}(B) = \cup_{x \in B} \pi_{\calK}(x)$. 
  
  We will use \thmref{ShortMarking} to analyze the closest-point projection
  in the Lipschitz metric to a cobounded \Teich geodesic $\GT$.
  Parametrizing $\GT$ by arc length (in the \Teich metric), we denote
  points along $\GT$ by $\GT(t)$. Along $\GT$, we have the following
  relationship between the hyperbolic length and the extremal length of a
  curve: 
  
  \begin{lemma}[\cite{minsky:PR}] \label{Lem:ExtHypThick}
     
     For any $x$ in the thick part of \T and any curve $\alpha$, 
     \[ 
        \ell_x(\alpha) \emul \sqrt{\Ext_x(\alpha)}.
     \]
   
  \end{lemma}   
  
  \noindent Furthermore, the length of $\alpha$ in either sense varies
  along $\GT(t)$ coarsely like $\cosh(t)$ \cite[Equation (2)]{rafi:HT}.
  Therefore, it makes sense to talk about a point $x_{t_\alpha} =
  \GT(t_\alpha)$ on which the length of $\alpha$ is minimal, and away from
  $x_{t_\alpha}$ in either direction the length of $\alpha$ grows
  exponentially. If there are several minimal points, then we choose
  $t_\alpha$ arbitrarily among them. We call $t_\alpha$ the \emph{balanced
  time} of $\alpha$.

  The first statement of the following lemma is a consequence of
  \cite[Lemma 3.3]{minsky:QP}. The second statement follows immediately
  from the first one and \lemref{ExtHypThick}.   
  
  \begin{lemma} \label{Lem:Minsky} 
     
     There exist constants $c_1$, $c_2$, and $D$, depending only on $S$,
     so that for any curves $\alpha$ and $\beta$ and any cobounded \Teich
     geodesic $\GT$, 
     \[ 
        |t_{\alpha} - t_{\beta}| \geq D
        \quad \Longrightarrow \quad
        \I(\alpha,\beta)^2 
        \geq 
        c_1 \, e^{2 |t_{\alpha} - t_{\beta}|} \,
        \Ext_{x_{t_{\alpha}}}(\alpha) \, \Ext_{x_{t_{\beta}}}(\beta) 
     \]
     and    
     \[ 
        |t_{\alpha} - t_{\beta}| \geq D
        \quad \Longrightarrow \quad
        \I(\alpha,\beta) 
        \geq 
        c_2 \, e^{|t_{\alpha} - t_{\beta}|} \, \ell_{x_{t_{\alpha}}}(\alpha) \,
        \ell_{x_{t_{\beta}}}(\beta). 
     \]
  \end{lemma}
  
  \begin{lemma} \label{Lem:Balanced}

     Let $\GT$ be a cobounded \Teich geodesic. Suppose $x\in \T$ is a point
     not on $\GT$ and $x_t \in \pi_{\GT}(x)$. Then for any $\alpha \in
     \mu_x$, we have $|t-t_{\alpha}|=O(1)$.  
  
  \end{lemma}
  
  \begin{proof} 
  
     Let $\beta \in \mu_x$ be a candidate curve from $x$ to $x_{t_\alpha}$.
     The curves $\alpha$ and $\beta$ have bounded intersection number, so
     by \lemref{Minsky}, $|t_\alpha-t_\beta | = O(1)$ (note that since
     $\GT$ is cobounded, the quantities $\ell_{x_{t_\alpha}}(\alpha)$ and
     $\ell_{x_{t_\beta}}(\beta)$ are bounded below). Away from $t_\beta$, the
     length of $\beta$ grows exponentially. We have
     \[ e^{\dL(x,x_t)} \ge
     \frac{\ell_{x_t}(\beta)}{\ell_x(\beta)} \gmul e^{\big(
     |t-t_\alpha|-|t_\alpha-t_\beta| \big)}
     \frac{\ell_{x_{t_\alpha}}(\beta)}{\ell_x(\beta)}.\] Taking $\log$ on both
     sides yields 
     \[ \dL(x,x_t) \gadd |t-t_\alpha| - |t_\alpha-t_\beta| +
     \dL(x,x_{t_\alpha}).\] Since $x_t$ is the closest-point projection of $x$ to
     $\GT$, $\dL(x,x_t) \le \dL(x, x_{t_\alpha})$. Together this implies
     $|t-t_\alpha| = O(1)$.
  \end{proof} 
  
  By a \emph{Lipschitz ball} of radius $R$ centered at $x$, we will mean the set 
  \[
    B_{L}(x,R) = \{ y \in \T \st \dL(x,y) \leq R \}.
  \]
  The following is a precise formulation of \thmref{IntroProjection}.   

  \begin{theorem}[Lipschitz projection to \Teich geodesics]
  \label{Thm:Projection} 
  
     There exists a constant $b$ depending only on $S$ such that, for any
     cobounded \Teich geodesic $\GT$, any $x \in \T$, and any
     constant $R<d_{L}(x,\GT)$, we have
     \[ 
        \diam_L \Big( \pi_{\GT} \big( B_L(x,R) \big) \Big) \leq b.
     \]
  
  \end{theorem} 
  
  \begin{proof}
      
     Let $y \in B_L(x,R)$, and let $\mu_x$ and $\mu_y$ be the associated
     short markings on $x$ and $y$ respectively. Let $x_t \in
     \pi_{\GT}(x)$. By  \lemref{Balanced}, we can choose $\alpha \in \mu_x$
     such that
     \[
       \dL(x, \GT)\eadd  \log \frac{\ell_{x_t}(\alpha)}{\ell_x(\alpha)},
     \]
     and  \thmref{ShortMarking} implies  
     \[ \log \frac{\ell_{x_t}(\alpha)}{\ell_x(\alpha)}\eadd  
     \log \frac{\ell_{x_{t_{\alpha}}}(\alpha)}{\ell_x(\alpha)},
     \]
     where $t_\alpha$ is the balance time for $\alpha$ along $\GT$. Hence 
     \begin{equation*}
       \dL(x, \GT)
       \eadd 
       \log \frac{\ell_{x_{t_\alpha}}(\alpha)}{\ell_x(\alpha)}.
     \end{equation*}
     Similarly, choose $\beta \in \mu_y$ so that
     \begin{equation*}
        \dL(y, \GT) 
        \eadd \log \frac{\ell_{x_{t_\beta}}(\beta)}{\ell_y(\beta)}. 
     \end{equation*}
     The theorem will hold if $|t_{\alpha}-t_{\beta}|$ is uniformly
     bounded. 
     
     Let $D$ be the constant of \lemref{Minsky}. If $|t_\alpha - t_\beta| <
     D$, then we are done. So suppose $|t_{\alpha} - t_{\beta}| \geq D$, in
     which case
     \[ 
        \I(\alpha,\beta) \gmul e^{|t_{\alpha}-t_{\beta}|}
        \, \ell_{x_{t_{\alpha}}}(\alpha) \, \ell_{x_{t_{\beta}}}(\beta).
     \]
    Since $\beta \in \mu_y$, by \propref{ShortEquality},
    $\ell_y(\alpha)\gmul i(\alpha,\beta)\ell_y(\bbeta)$. Therefore,
    \begin{align*} 
       e^{\dL(x,y)}
       \geq \frac{\ell_y(\alpha)}{\ell_x(\alpha)}
       & \gmul \frac{\I(\alpha,\beta)\, \ell_y(\bbeta)}{\ell_x(\alpha)} \\
       & \gmul \frac{e^{|t_{\alpha}-t_{\beta}|} \,
          \ell_{x_{t_{\alpha}}}(\alpha) \, \ell_{x_{t_{\beta}}}(\beta) \,
          \ell_y(\bbeta)}{\ell_x(\alpha)}. 
    \end{align*}
    Applying $\log$ to both sides above yields
    \begin{equation*} 
       \dL(x,y) 
       \gadd |t_{\alpha}-t_{\beta}| + \dL(x, \GT) +
       \log \big( \ell_{x_{t_\beta}}(\beta) \, \ell_y(\bbeta) \big).
     \end{equation*} 
    On the other hand, $\dL(x,y) \le R < \dL(x, \GT)$, so the proof will be
    complete if the product $\ell_{x_{t_\beta}}(\beta) \, \ell_y(\bbeta)$
    is bounded from below. Since $\GT$ is ($\ep'$)--cobounded, the length
    of every curve on $x_{t_\beta}$ is bounded below, so we only need to
    consider the situation when $\ell_y(\bbeta)$ is small
    (say $\ell_y(\bbeta)<\ep'$). In this case, since $\beta$ and $\bbeta$
    intersect, $\beta$ has to be long
    ($\ell_y(\beta)\gmul\log\frac{1}{\ep'}$). But $\beta$ is the candidate
    curve from $y$ to a point in $\pi_{\GT}(y)$ which we know is at most a
    bounded distance away from $x_{t_\beta}$. Thus, 
    \[ 
       \frac{\ell_{x_{t_{\beta}}}(\beta)}{\ell_y(\beta)}
       \gmul
       \frac{\ell_{x_{t_{\beta}}}(\bbeta)} {\ell_y(\bbeta)}.
    \] 
    We conclude 
    \[ 
      \ell_{x_{t_{\beta}}}(\beta) \, \ell_y(\bbeta) \gmul
      \ell_{x_{t_{\beta}}}(\bbeta) \, \ell_y(\beta)
      \gmul 1. \qedhere
    \] 
  \end{proof}
  
  \subsection*{Projection in the \Teich metric}

  We now sketch a short proof that the closest-point projection with
  respect to the \Teich metric to a cobounded \Teich geodesic is strongly
  contracting. This was first established by Minsky in \cite{minsky:QP}.
  This part is independent from the rest of the paper.
  
  Let $\Pi_\GT$ be the closest-point projection to $\GT$ with respect to
  the \Teich metric.

  \begin{theorem}[\cite{minsky:QP}]
     
     For any cobounded \Teich geodesic $\GT$ and for any \Teich ball $B$
     disjoint from $\GT$, $\diam_T \big( \Pi_\GT(B) \big)$ is uniformly
     bounded. 

  \end{theorem}

  \begin{proof} 
     
    As discussed before, \propref{AnyUpper} holds for extremal length as
    long as $x$ is in the thick part (see \remref{ExtUpper}). Therefore we
    have an analogue of \thmref{ShortMarking}: For any $x \in B$ and any
    $x_t \in \Pi_\GT(x)$, there exists a candidate curve $\alpha \in \mu_x$
    from $x$ to $x_t$. The same argument for \lemref{Balanced} will also
    show that $x_t$ is a bounded distance from $x_{t_\alpha}$. Replacing
    hyperbolic length by extremal length, we can carry out the same
    analysis as in \thmref{Projection} to finish the proof.
  \end{proof}

\section{Bounded projection to and stability of Lipschitz geodesics} 

  \label{Sec:Stable}

  In this section, we prove \thmref{IntroFellow} and
  \thmref{IntroContracting} of the introduction. Before we restate and
  prove the theorems, we first define what it means to fellow travel in the
  Lipschitz metric.
  
  Let $\GT(t) \from [0,d] \to \T$ and $\GL(t) \from [0,d] \to \T$ be
  respectively a \Teich and a Lipschitz geodesic parametrized by arc length
  (in their respective metric). We will say $\GL$ and $\GT$ fellow travel
  in the Lipschitz metric if there exists a constant $R$ depending only \s
  such that, for every $t \in [0,d]$, 
  \[ \max \Big\{ \dL \big( \GL(t), \GT(t) \big), \dL \big( \GT(t), \GL(t)
  \big) \Big\} \le R.\]
  
  \begin{theorem}[Lipschitz geodesic fellow travels \Teich
  geodesic]\label{Thm:Fellow}
     
     Suppose $x,y \in \T$ are thick and have bounded combinatorics. Then
     any Lipschitz geodesic $\GL$ from $x$ to $y$ is cobounded. In fact,
     $\GL$ fellow travels the \Teich geodesic with endpoints $x$ and $y$.
     More precisely, let $d=\dL(x,y)$ and let $\GT : \RR \to \T$ be the
     \Teich geodesic such that $\GT(0) = x$ and passing through $y$. Then
     $\GL \from [0,d] \to \T$ fellow travels $\GT \from [0,d] \to \T$.
  
  \end{theorem}
  
  By previous result in \thmref{Projection}, the Lipschitz closest-point
  projection to $\GT$ is strongly contracting. This implies that, if one
  moves along $\GL$, the rate of progress of the Lipschitz projection to
  $\GT$ is inversely proportional to the distance between $\GL$ and $\GT$.
  (A segment of length $R$ passing through a point $z$ that has distance
  $R$ from $\GT$ projects to a subset of $\GT$ with uniformly bounded
  size.) In order to apply a standard short-cut argument (see proof of
  \thmref{Fellow}), we need an additional fact about the asymmetry of $\dL$
  which is a corollary of \cite[Proposition 4.1]{rafi:TL}.

  \begin{lemma} \label{Lem:Asymmetry}

     Let $x \in \T$ be thick. Then there exists a constant $C$
     depending only on $S$ such that for any $y \in \T$
     \[ \dL(x, y) \le C \dL(y , x). \]

  \end{lemma}

  \begin{proof}

     From \cite[Proposition 4.1]{rafi:TL} we have (in \cite{rafi:TL} $\dL$
     is the symmetrized Lipschitz metric): 
     \begin{equation} \label{Eq:RC}
       \dT(x,y) \emul \max \{ \dL(x,y), \dL(y,x) \}.
     \end{equation} 
     By \eqnref{Kerck}, there is a curve $\alpha$ such that 
     $\dT(y,x) \emul \frac{1}{2} \log
     \frac{\Ext_y(\alpha)}{\Ext_x(\alpha)}.$ Since $x$ is thick, by
     \lemref{ExtHypThick}, $\Ext_x(\alpha)\emul \ell_x(\alpha)^2$. Since
     the extremal length is defined as a supremum over all metrics in a
     conformal class, we have $\Ext_y(\alpha) \gmul \ell_y(\alpha)^2$.
     Hence, 
     \[ \dL(y,x) \ge \log \frac{\ell_x(\alpha)}{\ell_y(\alpha)} \gmul
     \frac{1}{2} \log \frac{\Ext_x(\alpha)}{\Ext_y(\alpha)} \emul \dT(y,x). \]
     Also by \eqnref{RC}, $\dT(x,y) \gmul \dL(x,y)$. The lemma follows from
     the symmetry of the \Teich metric.
  \end{proof}

  \begin{proof}[Proof of \thmref{Fellow}]
    
    By assumption, $x$ and $y$ have bounded combinatorics, thus $\GT$ is
    cobounded by \thmref{Teich}. We will first show that there exists $R$
    such that, for any $x \in \GL$, there exists $x' \in \GT$ with
    $\dL(x,x') \le R$. That is, $\GL$ is contained in an $R$ Lipschitz
    neighborhood of $\GT$.

    For any $r > 0$, suppose a subinterval $[\overline{x}, \overline{y}]
    \subset \GL$ is such that $\dL(\overline{x}, \GT ) = \dL( \overline{y},
    \GT) = r$,  but $\dL(\overline{z}, \GT) > r$ for all other points
    $\overline{z} \in [\overline{x}, \overline{y}]$. By cutting
    $[\overline{x},\overline{y}]$ into segments of length at most $r$ and
    projecting each piece to $\GT$, we have
     \[
        \dL \big( \pi_{\GT}(\overline{x}), \pi_{\GT}(\overline{y}) \big)
        \le \frac{b}{r} \, \dL(\overline{x},\overline{y}) + b,
     \]
     where $b$ is the constant of \thmref{Projection}. Now fix $r=2b$. By
     the triangle inequality, 
     \begin{align*}
        \dL(\overline{x},\overline{y})
        & \le \dL \big( \overline{x}, \pi_{\GT}( \overline{x}) \big) 
        + \dL \big(\pi_{\GT}(\overline{x}), \pi_{\GT}(\overline{y}) \big) 
        + \dL \big(\pi_{\GT}(\overline{y}), \overline{y} \big) \\
        & \le r + \left( \frac{b}{r} \dL(\overline{x}, \overline{y}) + b
        \right) + C \dL \big(\overline{y} , \pi_{\GT}(\overline{y}) \big) \\
        & \le 2b + \left( \frac{1}{2} \dL(\overline{x}, \overline{y}) + b 
        \right) + C2b.
     \end{align*}
     We obtain $\dL(\overline{x},\overline{y}) \le 6b + 4Cb$. Therefore,
     any $\overline{z} \in [\overline{x},\overline{y}]$ is contained in an
     $R=8b+4Cb$ Lipschitz neighborhood of $\GT$. In view of
     \lemref{Asymmetry}, this also shows that $\GT$ is contained in a $CR$
     Lipschitz neighborhood of $\GL$. In particular, $\GL$ is cobounded
     (for some constant depending only on \s).      

     Now parametrize $x_t = \GL(t)$ and $y_t = \GT(t)$ such that $x=\GL(0)
     = \GT(0)$. We have shown that for any $t \in [0,d]$, $d=\dL(x,y)$,
     there exists $s$ such that $\dL \big( x_t, y_s \big) \le R$. The proof
     will be complete if $s \eadd t$. We have:
     \[
        s = \dT(x,y_s) \eadd \dL(x,y_s) \eadd \dL(x, x_t) = t.
     \]
     Thus for every $t \in [0,d]$, we have $\dL(x_t,y_t) \ladd 1$. The same
     thing is true for $\dL(y_t,x_t)$ since $\GL$ is cobounded.
  \end{proof}

  We now show that the closest-point projection to $\GL$ is also strongly
  contracting. As a corollary, $\GL$ is stable. The precise formulations
  are below.

  \begin{theorem}[Bounded projection to Lipschitz geodesics]
  \label{Thm:Contracting}
    
     Suppose $x, y \in \T$ are thick and have bounded combinatorics. There
     exists a constant $R$ such that whenever $\GL$ is a Lipschitz geodesic
     from $x$ to $y$ and $B$ is a Lipschitz ball with 
     \[ 
        \dL (B, \GL) = \min_{z \in B} \dL(z, \GL) > R, 
     \]
     then the Lipschitz projection of $B$ to $\GL$ is uniformly bounded.
     
  \end{theorem}
  
  \begin{proof}
     
     Let $\GT$ be the \Teich geodesic from $x$ to $y$. Let $R$ be the
     minimum constant such that $\GL$ is contained in the $R$ Lipschitz
     neighborhood of $\GT$ (\thmref{Fellow}). With this $R$, any Lipschitz
     ball $B$ satisfying the criterion of the theorem is disjoint from
     $\GT$. Therefore, by \thmref{Projection}, the projection of $B$ to
     $\GT$ has uniformly bounded diameter. To see that the projection of
     $B$ to $\GL$ also has uniformly bounded diameter, it suffices to show
     that, for any $z \in B$, the distance between $\pi_\GT \circ
     \pi_\GL(z)$ and $\pi_{\GT}(z)$ is uniformly bounded. 
     
     We refer to \figref{Contracting} for this proof. By \lemref{Balanced},
     $\pi_{\GT}(z)$ is uniformly bounded from $x_{t_\alpha} =
     \GT(t_\alpha)$, where $\alpha \in \mu_z$ is a candidate curve for the
     Lipschitz distance from $z$ to $\GT$, and $t_\alpha$ is the balanced
     time for $\alpha$. Now let $w \in \pi_\GL(z)$ and let $x_t \in
     \pi_{\GT}(w)$. We will show $|t_\alpha - t|$ is uniformly bounded.
     Choose a point $w' \in \GL$ so that $\dL(w', x_{t_\alpha})$ is
     minimal. In particular, $\dL(w', x_{t_\alpha}) \le R$, and 
     \begin{equation} \label{Eq:Upper} 
        \dL(z, w) \le \dL(z, w') \le \dL(z, x_{t_\alpha}) + CR,
     \end{equation} 
     where $C$ is the constant of \lemref{Asymmetry}. On the other hand, 
     \begin{align} \label{Eq:Lower} 
        \dL(z, w) 
        & \ge \log \frac{\ell_w(\alpha)}{\ell_z(\alpha)}\\ 
        & = \log \frac{\ell_w(\alpha)}{\ell_{x_t}(\alpha)} + \log
        \frac{\ell_{x_t}(\alpha)}{\ell_{x_{t_\alpha}}(\alpha)} + \log
        \frac{\ell_{x_{t_\alpha}}(\alpha)}{\ell_z(\alpha)} \nonumber \\ 
        & \gadd \log \frac{\ell_w(\alpha)}{\ell_{x_t}(\alpha)} + |t_\alpha - t| 
        + \dL(z, x_{t_\alpha}). \nonumber 
    \end{align} 
    Since $x_t \in \pi_{\GT}(w)$, $\dL(w,x_t) \le R$. Hence, 
    \[ \log \frac{\ell_w(\alpha)}{\ell_{x_t}(\alpha)} = - \log
    \frac{\ell_{x_t}(\alpha)}{\ell_w(\alpha)} \ge - \dL(w,x_t) \ge -R, \]
    Putting this together with \eqnref{Upper} and \eqnref{Lower}
    yields $|t_\alpha-t| \ladd (C+1)R$.  
  \end{proof}
  
  \begin{figure}[ht]
    \setlength{\unitlength}{3.6pt}
    \begin{picture}(90,45)
     \put(.6,.8){\includegraphics[width=130\unitlength]{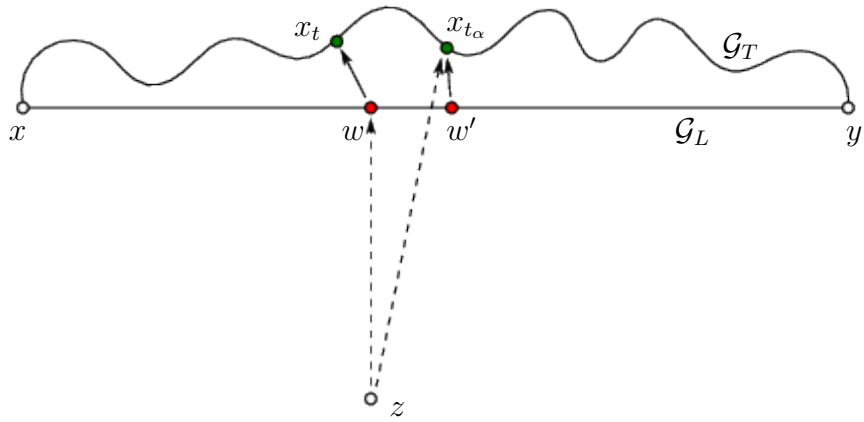}}
     \put(0,29)   {$x$}
     \put(88,29)  {$y$}
     \put(35,29)  {$w$}
     \put(46,29)  {$w'$}
     \put(30,40)  {$x_t$}
     \put(46,40.5){$x_{t_\alpha}$}
     \put(40,0)   {$z$}
     \put(70,29)  {$\GL$}
     \put(75,38)  {$\GT$}
    \end{picture}
     \caption{Bounded projection to Lipschitz geodesics} 
     \label{Fig:Contracting}
  \end{figure}
  
  \begin{corollary}[Stability of Lipschitz geodesics]\label{Cor:Stable}
     
     Suppose $x, y \in \T$ are thick and have bounded combinatorics. Then
     any Lipschitz quasi-geodesic from $x$ to $y$ (after reparametrization)
     fellow travels any Lipschitz geodesic from $x$ to $y$.  
  
  \end{corollary}
  
  \begin{proof}

     The same argument in the proof of \thmref{Fellow} can be applied here.
     Except now $\GL$ will play the role of $\GT$, and any Lipschitz
     quasi-geodesic from $x$ to $y$ will play the role of $\GL$. 
  \end{proof}
  
  We remark that, in general, a Lipschitz geodesic from $x$ to $y$ is not a
  Lipschitz geodesic from $y$ to $x$, even after reparametrization. One
  does not even expect the Hausdorff distance between a geodesic from $x$
  to $y$ and a geodesic from $y$ to $x$ to be bounded. (The Hausdorff
  distance is the smallest $R$ such that each is contained in an $R$
  Lipschitz neighborhood of the other). However, the notion of bounded
  combinatorics is a symmetric notion, as it is defined using distances in
  curve graphs. Since \Teich geodesics are independent of the order of the
  endpoints, we can also deduce the following corollary. 
  
  \begin{corollary}
     
     Suppose $x, y \in \T$ are thick and have bounded combinatorics. 
     Then the Hausdorff distance between any Lipschitz 
     geodesic from $x$ to $y$ and any Lipschitz geodesic 
     from $y$ to $x$ is uniformly bounded. 

  \end{corollary}

 
  \bibliographystyle{alpha}
  \bibliography{main}

  \end{document}